\documentclass[11pt, twoside]{article}
\usepackage{latexsym}
\usepackage{amsmath}
\usepackage{amssymb}
\usepackage[all]{xy}
\usepackage{amsfonts}
\usepackage{verbatim}
\usepackage{amsthm}
\usepackage{mathrsfs}
\usepackage{epsfig}
\usepackage{xy}
\usepackage{array}
\usepackage{stmaryrd}
\usepackage{graphicx,color}
\usepackage{xcolor}
\usepackage{tikz}
\usetikzlibrary{arrows,calc}
\usepackage{etex}
\usepackage{mathdots}
\usepackage{float}
\usepackage{graphics}
\usepackage{pdflscape}
\usepackage{extarrows}
\usepackage{anysize,hyperref}
\input xypic
\xyoption{all}
\usepackage{bm}
\usepackage[perpage,symbol]{footmisc}
\usepackage{setspace}
\renewcommand{\cal}{\mathcal}
\def\A{\mathcal{A}}
\def\T{\mathcal{T}}

\def\P{\mathcal{P}}
\def\I{\mathcal{I}}

\def\C{\mathscr{C}}
\def\E{\mathbb{E}}
\def\F{\mathbb{F}}
\def\s{\mathfrak{s}}
\def\id{\mathrm{id}}
\def\op{^\mathrm{op}}
\def\Ab{\mathsf{Ab}}
\def\del{\delta}
\def\dr{\ar@{->}[r]}

\def\X{\mathscr{X}}

\def\add{\mbox{add}}
\def\Ext{\mbox{Ext}}
\def\Hom{\mbox{Hom}}

\def\sets{\mathsf{Sets}\hspace{.01in}}
\newcommand{\CC}{{\bf{C}}^{n+2}_{\C}}
\newcommand{\mr}{\hbox{\boldmath$\cdot$}}
\newcommand{\ov}{\overset}
\newcommand{\lra}{\longrightarrow}
\newcommand{\co}{\colon}
\newcommand{\uas}{^{\ast}}            %%% ^*
\newcommand{\sas}{_{\ast}}
\newcommand{\Xd}{\langle X^{\mr},\del\rangle}  %%% <X,¦Ä>
\newcommand{\Yr}{\langle Y^{\mr},\rho\rangle}  %%% <Y,¦Ñ>
\newcommand{\ush}{^\sharp}           %%% ^sharp
\newcommand{\ssh}{_\sharp}
\SelectTips{cm}{10}

\begin{document}
\baselineskip=15pt
\title{\Large{\bf Frobenius $\bm{n}$-exangulated categories}}
\medskip
\author{Yu Liu and Panyue Zhou\footnote{Corresponding author.~Yu Liu is supported by the Fundamental Research Funds for the Central Universities (Grant No. 2682019CX51) and the National Natural Science Foundation of China (Grants No. 11901479). Panyue Zhou is supported by the National Natural Science Foundation of China (Grant Nos. 11901190 and 11671221), and the Hunan Provincial Natural Science Foundation of China (Grant Nos. 2018JJ3205),  and by the Scientific Research Fund of Hunan Provincial Education Department (Grant No. 19B239).}}

\date{}

\maketitle
\def\blue{\color{blue}}
\def\red{\color{red}}

\newtheorem{theorem}{Theorem}[section]
\newtheorem{lemma}[theorem]{Lemma}
\newtheorem{corollary}[theorem]{Corollary}
\newtheorem{proposition}[theorem]{Proposition}
\newtheorem{conjecture}{Conjecture}
\theoremstyle{definition}
\newtheorem{definition}[theorem]{Definition}
\newtheorem{question}[theorem]{Question}
\newtheorem{remark}[theorem]{Remark}
\newtheorem{remark*}[]{Remark}
\newtheorem{example}[theorem]{Example}
\newtheorem{example*}[]{Example}
\newtheorem{condition}[theorem]{Condition}
\newtheorem{condition*}[]{Condition}
\newtheorem{construction}[theorem]{Construction}
\newtheorem{construction*}[]{Construction}

\newtheorem{assumption}[theorem]{Assumption}
\newtheorem{assumption*}[]{Assumption}

\baselineskip=17pt
\parindent=0.5cm

\begin{abstract}
\baselineskip=16pt
Herschend--Liu--Nakaoka introduced the notion of $n$-exangulated
categories as higher dimensional analogues of extriangulated categories defined by Nakaoka--Palu.
The class of $n$-exangulated categories contains
$n$-exact categories and $(n+2)$-angulated categories as examples.
In this article,  we introduce a
notion of Frobenius $n$-exangulated categories
which are a generalization of Frobenius $n$-exact categories. We
show that the stable category of a Frobenius $n$-exangulated category is an
$(n+2)$-angulated category. As an application, this result
generalizes the work by Jasso.  We provide a class of $n$-exangulated categories which
are neither $n$-exact categories nor $(n+2)$-angulated categories.  Finally, we discuss an application of the main results and give some examples illustrating it.\\[0.5cm]
\medskip
\end{abstract}

\pagestyle{myheadings}
\markboth{\rightline {\scriptsize   Y. Liu and P. Zhou}}
         {\leftline{\scriptsize  Frobenius $n$-exangulated categories}}

\section{Introduction}
Higher homological algebra was introduced by Iyama \cite{I}, and it deals with $n$-cluster tilting subcategories of abelian categories (resp. exact categories). All short exact sequences in such a subcategory are split, but it has nice exact sequences with $n+2$ objects. This was recently formalized by Jasso \cite{J} in the theory of $n$-abelian categories (resp. $n$-exact categories).
There exists also a derived version of the theory focusing on $n$-cluster tilting subcategories of triangulated categories as introduced by Gei{\ss}, Keller, and Oppermann in the theory of $(n+2)$-angulated categories in \cite{GKO}. The properties of $(n+2)$-angulated categories have been investigated by  Bergh--Thaule in \cite{BT}. Setting $n=1$ recovers the notions of abelian, exact and
triangulated categories.

Extriangulated categories were recently introduced by Nakaoka and Palu \cite{NP} by
extracting those properties of $\Ext^1$ on exact categories and on triangulated categories
that seem relevant from the point of view of cotorsion pairs. In particular, exact
categories and triangulated categories are extriangulated categories. There are a lot of examples of extriangulated categories which are neither exact triangulated categories nor triangulated categories.
The data of such a category is a triplet $(\C,\E,\s)$, where $\C$ is an additive category, $\E\colon\C\op\times\C\to\Ab$ is an additive bifunctor and $\s$ assigns to each $\delta\in \E(C,A)$ a class of $3$-term sequences with end terms $A$ and $C$ such that certain axioms hold. Herschend--Liu--Nakaoka \cite{HLN} introduced an $n$-analogue of this notion called $n$-exangulated categories. Such a category is a similar triplet $(\C,\E,\s)$, with the main distinction being that the $3$-term sequences mentioned above are replaced by $(n+2)$-term sequences. We note that the case $n=1$ corresponds to extriangulated categories.
As typical examples we have that $n$-exact and $(n+2)$-angulated categories are $n$-exangulated, see \cite[Proposition 4.34 and Proposition 4.5]{HLN}.
In an extriangulated category $(\C,\E,\s)$  that has enough projectives and
injectives, Liu--Nakaoka \cite{LN} introduced higher extension
groups $\E^i$ using dimension shift, and defined the notion
of an $n$-cluster tilting subcategory $\X$. Under suitable conditions,
 Herschend--Liu--Nakaoka \cite{HLN} showed that $\X$ is $n$-exangulated categories, see \cite[Theorem 5.41]{HLN}.
However, there are some other examples of $n$-exangulated categories which are neither $n$-exact nor $(n+2)$-angulated, see \cite[Section 6]{HLN}.

Motivated by the definitions of Frobenius $n$-exact categories and Frobenius extriangulated categories, we define Frobenius $n$-exangulated categories.
These are $n$-exangulated categories with enough
projectives and enough injectives, and such that these two classes of objects coincide.  Frobenius $n$-exangulated categories are related to $(n+2)$-angulated categories. We now state the first main result in this article.

\begin{theorem}\emph{(see Theorem \ref{main1} for details)}
Let $\C$ be a Frobenius $n$-exangulated category.
Then the stable category $\overline{\C}$ is an $(n+2)$-angulated category.
\end{theorem}

This generalizes a result by Jasso \cite[Theorem 5.11]{J} for Frobenius $n$-exact categories.

In order to give the second main result in this article, we need the following some notions.
Let $(\C,\E,\s)$ be an $n$-exangulated category and $\X$ be a subcategory of $\C$.
For each pair of objects $A$ and $C$ in $\C$, define
$$\F^{\X}(C,A)=
\{A_0\xrightarrow{f}A_1\xrightarrow{}A_2\xrightarrow{}\cdots\xrightarrow{}A_{n-1}
\xrightarrow{}A_n\xrightarrow{}A_{n+1}\overset{\delta}{\dashrightarrow}
\,\mid\, f\ \textrm{is an}\ \X\textrm{-monic}\}.$$
Dually, we define for each pair of objects $A$ and $C$ in $\C$
$$\F_{\X}(C,A)=
\{A_0\xrightarrow{}A_1\xrightarrow{}A_2\xrightarrow{}\cdots\xrightarrow{}A_{n-1}
\xrightarrow{}A_n\xrightarrow{g}A_{n+1}\overset{\delta}{\dashrightarrow}
\,\mid\, g\ \textrm{is an}\ \X\textrm{-epic}\}.$$

\begin{theorem}\emph{(see Theorem \ref{main2} for details)}
Let $(\C,\E,\s)$ be an $n$-exangulated category.
If $\X$ is a strongly functorially finite subcategory of $\C$, then $(\C,\F,\s_{\F})$ is a Frobenius $n$-exangulated category whose projective-injective objects are precisely $\X$, where $\F:=\F^{\X}\cap~\F_{\X}$.
\end{theorem}

The above construction gives $n$-exangulated categories which are not
$n$-exact nor $(n+2)$-angulated in general.

This article is organized as follows. In Section 2, we review some elementary definitions
and facts on $n$-exangulated categories. In Section 3, we prove our first main
result. In Section 4, we prove our second main result. In Section 5, we discuss an application of main results and  give some examples illustrating our main results.

\section{Preliminaries}
Let us briefly recall some definitions and basic properties of $n$-exangulated categories from \cite{HLN}.
Throughout this article, let $\C$ be an additive category and $n$ be any positive integer.

\begin{definition}\cite[Definition 2.1]{HLN}
Suppose that $\C$ is equipped with an additive bifunctor $\E\colon\C\op\times\C\to\Ab$, where $\Ab$ is the category of abelian groups. For any pair of objects $A,C\in\C$, an element $\del\in\E(C,A)$ is called an {\it $\E$-extension} or simply an {\it extension}. We also write such $\del$ as ${}_A\del_C$ when we indicate $A$ and $C$.

Let ${}_A\del_C$ be any extension. Since $\E$ is a bifunctor, for any $a\in\C(A,A')$ and $c\in\C(C',C)$, we have extensions
$$ \E(C,a)(\del)\in\E(C,A')\ \ \text{and}\ \ \E(c,A)(\del)\in\E(C',A). $$
We abbreviately denote them by $a_{\ast}\del$ and $c^{\ast}\del$.
In this terminology, we have
$$\E(c,a)(\del)=c^{\ast}a_{\ast}\del=a_{\ast}c^{\ast}\del\in\E(C',A').$$
For any $A,C\in\C$, the zero element ${}_A0_C=0\in\E(C,A)$ is called the {\it split $\E$-extension}.
\end{definition}

\begin{definition}\cite[Definition 2.3]{HLN}
Let ${}_A\del_C,{}_{A'}\del'_{C'}$ be any pair of $\E$-extensions. A {\it morphism} $(a,c)\colon\del\to\del'$ of extensions is a pair of morphisms $a\in\C(A,B)$ and $c\in\C(A',C')$ in $\C$, satisfying the equality
$$a_{\ast}\del=c^{\ast}\del'. $$
\end{definition}

\begin{definition}\cite[Definition 2.7]{HLN}
Let $\bf{C}_{\C}$ be the category of complexes in $\C$. As its full subcategory, define $\CC$ to be the category of complexes in $\C$ whose components are zero in the degrees outside of $\{0,1,\ldots,n+1\}$. Namely, an object in $\CC$ is a complex $X^{\mr}=\{X_i,d^X_i\}$ of the form
\[ X_0\xrightarrow{d^X_0}X_1\xrightarrow{d^X_1}\cdots\xrightarrow{d^X_{n-1}}X_n\xrightarrow{d^X_n}X_{n+1}. \]
We write a morphism $f^{\mr}\co X^{\mr}\to Y^{\mr}$ simply $f^{\mr}=(f^0,f^1,\ldots,f^{n+1})$, only indicating the terms of degrees $0,\ldots,n+1$.
\end{definition}

\begin{definition}\cite[Definition 2.11]{HLN}
By Yoneda lemma, any extension $\del\in\E(C,A)$ induces natural transformations
\[ \del\ssh\colon\C(-,C)\Rightarrow\E(-,A)\ \ \text{and}\ \ \del\ush\colon\C(A,-)\Rightarrow\E(C,-). \]
For any $X\in\C$, these $(\del\ssh)_X$ and $\del\ush_X$ are given as follows.
\begin{enumerate}
\item[\rm(1)] $(\del\ssh)_X\colon\C(X,C)\to\E(X,A)\ ;\ f\mapsto f\uas\del$.
\item[\rm (2)] $\del\ush_X\colon\C(A,X)\to\E(C,X)\ ;\ g\mapsto g\sas\delta$.
\end{enumerate}
We abbreviately denote $(\del\ssh)_X(f)$ and $\del\ush_X(g)$ by $\del\ssh(f)$ and $\del\ush(g)$, respectively.
\end{definition}

\begin{definition}\cite[Definition 2.9]{HLN}
 Let $\C,\E,n$ be as before. Define a category $\AE:=\AE^{n+2}_{(\C,\E)}$ as follows.
\begin{enumerate}
\item[\rm(1)]  A pair $\Xd$ is an object of the category $\AE$ with $X^{\mr}\in\CC$
and $\del\in\E(X_{n+1},X_0)$ is called an $\E$-attached
complex of length $n+2$, if it satisfies
$$(d_0^X)_{\ast}\del=0~~\textrm{and}~~(d^X_n)^{\ast}\del=0.$$
We also denote it by
$$X_0\xrightarrow{d_0^X}X_1\xrightarrow{d_1^X}\cdots\xrightarrow{d_{n-2}^X}X_{n-1}
\xrightarrow{d_{n-1}^X}X_n\xrightarrow{d_n^X}X_{n+1}\overset{\delta}{\dashrightarrow}$$
\item[\rm (2)]  For such pairs $\Xd$ and $\langle Y^{\mr},\rho\rangle$,  $f^{\mr}\colon\Xd\to\langle Y^{\mr},\rho\rangle$ is
defined to be a morphism in $\AE$ if it satisfies $(f_0)_{\ast}\del=(f_{n+1})^{\ast}\rho$.

\end{enumerate}
\end{definition}

\begin{definition}\cite[Definition 2.13]{HLN}
 An {\it $n$-exangle} is an object $\Xd$ in $\AE$ that satisfies the listed conditions.
\begin{enumerate}
\item[\rm (1)] The following sequence of functors $\C\op\to\Ab$ is exact.
$$
\C(-,X_0)\xLongrightarrow{\C(-,\ d^X_0)}\cdots\xLongrightarrow{\C(-,\ d^X_n)}\C(-,X_{n+1})\xLongrightarrow{~\del\ssh~}\E(-,X_0)
$$
\item[\rm (2)] The following sequence of functors $\C\to\Ab$ is exact.
$$
\C(X_{n+1},-)\xLongrightarrow{\C(d^X_n,\ -)}\cdots\xLongrightarrow{\C(d^X_0,\ -)}\C(X_0,-)\xLongrightarrow{~\del\ush~}\E(X_{n+1},-)
$$
\end{enumerate}
In particular any $n$-exangle is an object in $\AE$.
A {\it morphism of $n$-exangles} simply means a morphism in $\AE$. Thus $n$-exangles form a full subcategory of $\AE$.
\end{definition}

\begin{definition}\cite[Definition 2.22]{HLN}
Let $\s$ be a correspondence which associates a homotopic equivalence class $\s(\del)=[{}_AX^{\mr}_C]$ to each extension $\del={}_A\del_C$. Such $\s$ is called a {\it realization} of $\E$ if it satisfies the following condition for any $\s(\del)=[X^{\mr}]$ and any $\s(\rho)=[Y^{\mr}]$.
\begin{itemize}
\item[{\rm (R0)}] For any morphism of extensions $(a,c)\co\del\to\rho$, there exists a morphism $f^{\mr}\in\CC(X^{\mr},Y^{\mr})$ of the form $f^{\mr}=(a,f_1,\ldots,f_n,c)$. Such $f^{\mr}$ is called a {\it lift} of $(a,c)$.
\end{itemize}
In such a case, we simple say that \lq\lq$X^{\mr}$ realizes $\del$" whenever they satisfy $\s(\del)=[X^{\mr}]$.

Moreover, a realization $\s$ of $\E$ is said to be {\it exact} if it satisfies the following conditions.
\begin{itemize}
\item[{\rm (R1)}] For any $\s(\del)=[X^{\mr}]$, the pair $\Xd$ is an $n$-exangle.
\item[{\rm (R2)}] For any $A\in\C$, the zero element ${}_A0_0=0\in\E(0,A)$ satisfies
\[ \s({}_A0_0)=[A\ov{\id_A}{\lra}A\to0\to\cdots\to0\to0]. \]
Dually, $\s({}_00_A)=[0\to0\to\cdots\to0\to A\ov{\id_A}{\lra}A]$ holds for any $A\in\C$.
\end{itemize}
Note that the above condition {\rm (R1)} does not depend on representatives of the class $[X^{\mr}]$.
\end{definition}

\begin{definition}\cite[Definition 2.23]{HLN}
Let $\s$ be an exact realization of $\E$.
\begin{enumerate}
\item[\rm (1)] An $n$-exangle $\Xd$ is called a $\s$-{\it distinguished} $n$-exangle if it satisfies $\s(\del)=[X^{\mr}]$. We often simply say {\it distinguished $n$-exangle} when $\s$ is clear from the context.
\item[\rm (2)]  An object $X^{\mr}\in\CC$ is called an {\it $\s$-conflation} or simply a {\it conflation} if it realizes some extension $\del\in\E(X_{n+1},X_0)$.
\item[\rm (3)]  A morphism $f$ in $\C$ is called an {\it $\s$-inflation} or simply an {\it inflation} if it admits some conflation $X^{\mr}\in\CC$ satisfying $d_X^0=f$.
\item[\rm (4)]  A morphism $g$ in $\C$ is called an {\it $\s$-deflation} or simply a {\it deflation} if it admits some conflation $X^{\mr}\in\CC$ satisfying $d_X^n=g$.
\end{enumerate}
\end{definition}

\begin{definition}\cite[Definition 2.27]{HLN}
For a morphism $f^{\mr}\in\CC(X^{\mr},Y^{\mr})$ satisfying $f^0=\id_A$ for some $A=X_0=Y_0$, its {\it mapping cone} $M_f^{\mr}\in\CC$ is defined to be the complex
\[ X_1\xrightarrow{d^{M_f}_0}X_2\oplus Y_1\xrightarrow{d^{M_f}_1}X_3\oplus Y_2\xrightarrow{d^{M_f}_2}\cdots\xrightarrow{d^{M_f}_{n-1}}X_{n+1}\oplus Y_n\xrightarrow{d^{M_f}_n}Y_{n+1} \]
where $d^{M_f}_0=\begin{bmatrix}-d^X_1\\ f_1\end{bmatrix},$
$d^{M_f}_i=\begin{bmatrix}-d^X_{i+1}&0\\ f_{i+1}&d^Y_i\end{bmatrix}\ (1\le i\le n-1),$
$d^{M_f}_n=\begin{bmatrix}f_{n+1}&d^Y_n\end{bmatrix}$.

{\it The mapping cocone} is defined dually, for morphisms $h^{\mr}$ in $\CC$ satisfying $h_{n+1}=\id$.
\end{definition}

\begin{definition}\cite[Definition 2.32]{HLN}
An {\it $n$-exangulated category} is a triplet $(\C,\E,\s)$ of additive category $\C$, additive bifunctor $\E\co\C\op\times\C\to\Ab$, and its exact realization $\s$, satisfying the following conditions.
\begin{itemize}
\item[{\rm (EA1)}] Let $A\ov{f}{\lra}B\ov{g}{\lra}C$ be any sequence of morphisms in $\C$. If both $f$ and $g$ are inflations, then so is $g\circ f$. Dually, if $f$ and $g$ are deflations then so is $g\circ f$.

\item[{\rm (EA2)}] For $\rho\in\E(D,A)$ and $c\in\C(C,D)$, let ${}_A\langle X^{\mr},c\uas\rho\rangle_C$ and ${}_A\Yr_D$ be distinguished $n$-exangles. Then $(\id_A,c)$ has a {\it good lift} $f^{\mr}$, in the sense that its mapping cone gives a distinguished $n$-exangle $\langle M^{\mr}_f,(d^X_0)\sas\rho\rangle$.
\item[{\rm (EA2$\op$)}] Dual of {\rm (EA2)}.
\end{itemize}
Note that the case $n=1$, a triplet $(\C,\E,\s)$ is a  $1$-exangulated category if and only if it is an extriangulated category, see \cite[Proposition 4.3]{HLN}.
\end{definition}

\begin{example}
From \cite[Proposition 4.34]{HLN} and \cite[Proposition 4.5]{HLN},  we know that $n$-exact categories and $(n+2)$-angulated categories are $n$-exangulated categories.
There are some other examples of $n$-exangulated categories
 which are neither $n$-exact nor $(n+2)$-angulated, see \cite[Section 6]{HLN} for more details.
\end{example}

\begin{lemma}\label{a1}
Let $(\C,\E,\s)$ be an $n$-exangulated category, and
$$A_0\xrightarrow{\alpha_0}A_1\xrightarrow{\alpha_1}A_2\xrightarrow{\alpha_2}\cdots\xrightarrow{\alpha_{n-2}}A_{n-1}
\xrightarrow{\alpha_{n-1}}A_n\xrightarrow{\alpha_n}A_{n+1}\overset{\delta}{\dashrightarrow}$$
a distinguished $n$-exangle. Then we have the following long exact sequences:
$$\C(-, A_0)\xrightarrow{}\C(-, A_1)\xrightarrow{}\cdots\xrightarrow{}
\C(-, A_{n+1})\xrightarrow{}\E(-, A_{0})\xrightarrow{}\E(-, A_{1})\xrightarrow{}\E(-, A_{2});$$
$$\C(A_{n+1},-)\xrightarrow{}\C(A_{n},-)\xrightarrow{}\cdots\xrightarrow{}
\C(A_0,-)\xrightarrow{}\E(A_{n+1},-)\xrightarrow{}\E(A_{n},-)\xrightarrow{}\E(A_{n-1},-).$$
\end{lemma}

\proof This follows from Definition 2.13 and Corollary 3.11 in \cite{HLN}.  \qed

\begin{definition}\cite[Definition 3.6 and Definition 3.7]{HLN}
Let $\C$ be an additive category and $\E\colon \C\op\times\C\to\Ab$ be  an additive
bifunctor.
\begin{enumerate}
\item[\rm (1)] A functor $\F\colon \C\op\times\C\to\sets$ is called a \emph{subfunctor} of $\E$ if it satisfies the following conditions.
\begin{itemize}
\item $\F(C,A)$ is a subset of $\E(C,A)$, for any $A,C\in\C$.

\item $\F(c,a)=\E(c,a)|_{\F(C,A)}$ holds, for any $a\in\C(A,A')$ and $c\in\C(C',C)$.
In this case, we write as $\F\subseteq\E$.
\end{itemize}

\item[\rm (2)] A subfunctor $\F\subseteq\E$ is said to be an \emph{additive subfunctor}
if $\F(C,A)\subseteq\E(C,A)$ is an abelian subgroup for any $A,C\in\C$.
In this case, $\F\colon \C\op\times\C\to\Ab$
itself becomes a additive bifunctor.

\item[\rm (3)] Let $\F\subseteq\E$ be an additive subfunctor. For a realization $\s$ of $\E$,
define $\s|_{\F}$ to be the restriction of s onto $\F$. Namely, it is defined by $\s|_{\F}(\delta)=\s(\delta)$ for any $\F$-extension $\delta$.
\end{enumerate}
\end{definition}

\begin{lemma}\emph{\cite[Proposition 3.14]{HLN}}\label{a6}
Let $(\C,\E,\s)$ be an $n$-exangulated category. For any additive subfunctor $\mathbb{F}\subseteq\E$, the following statements are equivalent.
\begin{itemize}
\item[\rm (1)] $(\C,\mathbb{F},\s_{\mathbb{F}})$ is $n$-exangulated.

\item[\rm (2)] $\s_{\mathbb{F}}$-inflations are  closed under composition.

\item[\rm (3)] $\s_{\mathbb{F}}$-deflations are closed under composition.
\end{itemize}
\end{lemma}

\section{Frobenius $n$-exangulated categories}
In this section, we introduce a
notion of Frobenius $n$-exangulated categories
which are a generalization of Frobenius $n$-exact categories. Moreover, we
prove that the stable category of a Frobenius $n$-exangulated category is an
$(n+2)$-angulated category.

Let $\C$ be  an additive category,
and $\X$ be a subcategory of $\C$.
Recall that we say a morphism $f\colon A \to B$ in $\C$ is an $\X$-\emph{monic} if
$$\Hom_{\C}(f,X)\colon \Hom_{\C}(B,X) \to \Hom_{\C}(A,X)$$
is an epimorphism for all $X\in\X$. We say that $f$ is an $\X$-\emph{epic} if
$$\Hom_{\C}(X,f)\colon \Hom_{\C}(X,A) \to \Hom_{\C}(X,B)$$
is an epimorphism for all $X\in\X$.
Similarly,
we say that $f$ is a left $\X$-approximation of $B$ if $f$ is an $\X$-monic and $A\in\X$.
We say that $f$ is a right $\X$-approximation of $A$ if $f$ is an $\X$-epic and $B\in\X$.

A subcategory $\X$ is called \emph{contravariantly finite} if any object in $\C$ admits a right
$\X$-approximation. Dually we can define  \emph{covariantly finite} subcategory.
A contravariantly finite and covariantly finite subcategory is called \emph{functorially finite}.
\begin{definition}\label{dd1}
Let $(\C,\E,\s)$ be an $n$-exangulated category. A subcategory $\X$ of $\C$ is called
\emph{strongly contravariantly finite}, if for any object $C\in\C$, there exists a distinguished $n$-exangle
$$B\xrightarrow{}X_1\xrightarrow{}X_2\xrightarrow{}\cdots\xrightarrow{}X_{n-1}\xrightarrow{}X_{n}\xrightarrow{~g~}C\overset{}{\dashrightarrow}$$
where $g$ is a right $\X$-approximation of $C$ and $X_i\in\X$.
Dually we can define \emph{strongly covariantly finite} subcategory.

A strongly contravariantly finite and strongly  covariantly finite subcategory is called \emph{ strongly functorially finite}.
\end{definition}

\begin{definition}\label{def2}
Let $(\C,\E,\s)$ be an $n$-exangulated category.
\begin{itemize}
\item[(1)] An object $P\in\C$ is called \emph{projective} if, for any distinguished $n$-exangle
$$A_0\xrightarrow{\alpha_0}A_1\xrightarrow{\alpha_1}A_2\xrightarrow{\alpha_2}\cdots\xrightarrow{\alpha_{n-2}}A_{n-1}
\xrightarrow{\alpha_{n-1}}A_n\xrightarrow{\alpha_n}A_{n+1}\overset{\delta}{\dashrightarrow}$$
and any morphism $c$ in $\C(P,A_{n+1})$, there exists a morphism $b\in\C(P,A_n)$ satisfying $\alpha_n\circ b=c$.
We denote the full subcategory of projective objects in $\C$ by $\P$.
Dually, the full subcategory of injective objects in $\C$ is denoted by $\I$.

\item[(2)] We say that $\C$ {\it has enough  projectives} if
for any object $C\in\C$, there exists a distinguished $n$-exangle
$$B\xrightarrow{\alpha_0}P_1\xrightarrow{\alpha_1}P_2\xrightarrow{\alpha_2}\cdots\xrightarrow{\alpha_{n-2}}P_{n-1}
\xrightarrow{\alpha_{n-1}}P_n\xrightarrow{\alpha_n}C\overset{\delta}{\dashrightarrow}$$
satisfying $P_1,P_2,\cdots,P_n\in\P$. We can define the notion of having \emph{enough injectives} dually.

\item[(3)]  $\C$ is said to be {\it Frobenius} if $\C$ has enough projectives and enough injectives
and if moreover the projectives coincide with the injectives.
\end{itemize}
\end{definition}

\begin{remark}
~\begin{itemize}
\item[\rm (1)]  In the case $n=1$, these agree with
the usual definitions \cite[Definition 3.23, Definition 3.25 and Definition 7.1]{NP}.

\item[\rm (2)] If $(\C,\E,\s)$ is an $n$-exact category, then these agree with \cite[Definition 3.11, Definition 5.3 and Definition 5.5]{J}.

\item[\rm (3)] If $(\C,\E,\s)$ is an $(n+2)$-angulated category, then
$\P=\I$ consists of zero objects. Moreover it always has enough projectives and enough injectives.
\end{itemize}
\end{remark}

\begin{lemma}\label{a5}
Let $(\C,\E,\s)$ be an $n$-exangulated category. Then the
following statements are equivalent for an object $P\in\C$.
\begin{itemize}
\item[\rm (1)] $\E(P,A)=0$ for any $A\in\C$;

\item[\rm (2)] $P$ is projective;

\item[\rm (3)] Any distinguished $n$-exangle $A_0\xrightarrow{\alpha_0}A_1\xrightarrow{\alpha_1}A_2\xrightarrow{\alpha_2}\cdots\xrightarrow{\alpha_{n-2}}A_{n-1}
\xrightarrow{\alpha_{n-1}}A_n\xrightarrow{\alpha_n}P\overset{\delta}{\dashrightarrow}$ splits.
\end{itemize}
\end{lemma}

\proof  (1) $\Rightarrow$ (2).
For any distinguished $n$-exangle
$$A_0\xrightarrow{\alpha_0}A_1\xrightarrow{\alpha_1}A_2\xrightarrow{\alpha_2}\cdots\xrightarrow{\alpha_{n-2}}A_{n-1}
\xrightarrow{\alpha_{n-1}}A_n\xrightarrow{\alpha_n}A_{n+1}\overset{\delta}{\dashrightarrow}$$
and any morphism $c$ in $\C(P,A_{n+1})$, by Lemma \ref{a1}, we have the following exact sequence:
$$\C(P, A_{n+1})\xrightarrow{\C(P,~\alpha_n)}\C(P, A_{n+1})\xrightarrow{}\E(P, A_{0})=0.$$
So there exists a morphism $b\in\C(P,A_n)$ such that $\alpha_n\circ b=c$.
This shows that $P$ is projective.

(2) $\Rightarrow$ (3).  Since $P$ is projective, there exist a morphism $u\colon P\to A_n$
such that $\alpha_nu=1_{P}$. By \cite[Claim 2.15]{HLN}, we have $\delta=0$.
Hence $\delta$ splits.

(3) $\Rightarrow$ (1).  It follows from \cite[Claim 2.15]{HLN}.  \qed
\bigskip

Let $\C$ be an additive category.
For two objects $A,B$ in $\X$ denote by $\X(A,B)$ the subgroup of $\Hom_{\C}(A,B)$ consisting of those morphisms which factor through an object in $\X$. Denote by $\C/\X$ the \emph{quotient category} of $\C$ modulo $\X$: the objects are the same as the ones in $\C$, for two objects $A$ and $B$ the Hom space is given by the quotient group $\Hom_{\C}(A,B)/\X(A,B)$.
Note that the quotient category $\C/\X$ is an additive category. We denote $\overline{f}$ the image of $f\colon A\to B$ of $\C$ in $\C/\X$.
\medskip

From this point on we assume that $(\C,\E,\mathfrak{s})$ is a Frobenius $n$-exangulated category.
We refer to this category $\C/\I$ as the stable category of $\C$.
In keeping the convention of the classical theory, if $(\C,\E,\mathfrak{s})$ is a Frobenius $n$-exangulated category, then we denote its \textbf{stable category} by $\overline{\C}$.
Any object $X\in\C$ admits a distinguished $n$-exangle
$$\xymatrix{X \ar[r]^{i_0^X\;} &I_1^X \ar[r]^{i_1^X\;} \ar[r] &\cdot\cdot\cdot \ar[r]^{i_{n-1}^X} &I_n^X \ar[r]^{i_n^X} \ar[r] &SX \ar@{-->}[r]^-{\del^X} &}$$ where $I_k^X\in \I$ for $k\in \{1,2,...,n \}$, $i_0^X$ is a left $\I$-appromation of $X$ and $i_n^X$ is a right $\I$-appromation of $SX$. We call it a distinguished $\I(X)$-exangle.
\medskip

Our aim is to show that the stable category of a Frobenius $n$-exangulated category, has a natural
structure of an $(n+2)$-angulated category. We begin with the construction of an auto-equivalence
$S\colon\overline{\C}\to\overline{\C}$. We need the following some lemmas.

\begin{lemma}\label{iso}
For a morphism $f\colon X\to Y$, there exists a commutative diagram of a distinguished  $\I(X)$-exangle and a distinguished $\I(Y)$-exangle
$$\xymatrix{
X \ar[r]^{i_0^X} \ar@{}[dr]|{\circlearrowright} \ar[d]_f &I_1^X \ar[r]^{i_1^X} \ar@{}[dr]|{\circlearrowright} \ar[d] &\cdot\cdot\cdot \ar[r]^{i_{n-1}^X} \ar@{}[dr]|{\circlearrowright} &I_n^X \ar[d] \ar@{}[dr]|{\circlearrowright} \ar[r]^{i_n^X} &SX \ar[d]^{Sf} \ar@{-->}[r]^-{\del^X} &\\
Y \ar[r]_{i_0^Y} &I_1^Y \ar[r]_{i_1^Y} &\cdot\cdot\cdot \ar[r]_{i_{n-1}^Y} &I_n^Y \ar[r]_{i_n^Y}  &SY\ar@{-->}[r]^-{\del^Y} &,
}
$$
where the morphism $\overline {Sf}$ is unique in $\overline \C$.
\end{lemma}

\begin{proof}
Since $I_1^X$ and $I_1^Y$ are in $\I$ and $i_0^X$, $i_0^Y$ are left $\I$-approximations, we have the required commutative diagram. If there is another morphism $(Sf)'\colon SX\to SY$ such that $$f_*\del^X=((Sf)')^*\del^Y,$$ we have $Sf-(Sf)'$ factors through $i_n^Y$, which implies $\overline {Sf}=\overline {(Sf)'}$.
\end{proof}

For any object $X$, we fix a distinguished  $\I(X)$-exangle $$\xymatrix{X \ar[r]^{i_0^X} &I_1^X \ar[r]^{i_1^X} \ar[r] &\cdot\cdot\cdot \ar[r]^{i_{n-1}^X} &I_n^X \ar[r]^{i_n^X} \ar[r] &SX \ar@{-->}[r]^-{\del^X} &,}$$ then we have a functor $S$ on $\overline \C$ such that $S\colon X\to SX$ and $S\colon \overline f\to \overline {Sf}$.

Moreover, if for any object $X$, we fix another distinguished  $\I(X)$-exangle $$\xymatrix{X \ar[r]^{{i_0^X}'} &{I_1^X}' \ar[r]^{{i_1^X}'} \ar[r] &\cdot\cdot\cdot \ar[r]^{{i_{n-1}^X}'} &{I_n^X}' \ar[r]^{{i_n^X}'} \ar[r] &S'X \ar@{-->}[r]^-{(\del^X)'} &,}$$ we can define a functor $S'$ in the same way as a functor $S$. We have the following lemma.

\begin{lemma}\label{lemma}
There exists a natural isomorphism $\Phi$ from $S$ to $S'$.
\end{lemma}

\begin{proof}
For any object $X$, we have the following commutative diagram
$$\xymatrix{
X \ar[r]^{i_0^X} \ar@{=}[d] \ar@{}[dr]|{\circlearrowright} &I_1^X \ar[r]^{i_1^X} \ar@{}[dr]|{\circlearrowright} \ar[d] &\cdot\cdot\cdot \ar@{}[dr]|{\circlearrowright} \ar[r]^{i_{n-1}^X} &I_n^X \ar@{}[dr]|{\circlearrowright} \ar[d] \ar[r]^{i_n^X} &SX \ar[d]^{f^X} \ar@{-->}[r]^-{\del^X} &\\
X \ar[r]_{{i_0^X}'} &{I_1^X}' \ar[r]_{{i_1^X}'} &\cdot\cdot\cdot \ar[r]_{{i_{n-1}^X}'} &{I_n^X}' \ar[r]_{{i_n^X}'}  &S'X \ar@{-->}[r]^-{(\del^X)'} &.
}
$$
Denote $\overline {f^X}$ by $\Phi^X$, by Lemma \ref{iso}, we know that $\Phi^X$ is an isomorphism in $\overline \C$.

For any morphism $g:X\to Y$, we have the following two commutative diagrams
$$\xymatrix{
X \ar[r]^{i_0^X} \ar@{=}[d] \ar@{}[dr]|{\circlearrowright} &I_1^X \ar@{}[dr]|{\circlearrowright} \ar[r]^{i_1^X} \ar[d] &\cdot\cdot\cdot \ar@{}[dr]|{\circlearrowright} \ar[r]^{i_{n-1}^X} &I_n^X \ar@{}[dr]|{\circlearrowright} \ar[d] \ar[r]^{i_n^X} &SX \ar[d]^{f^X} \ar@{-->}[r]^-{\del^X} &\\
X \ar[r]^{{i_0^X}'} \ar@{}[dr]|{\circlearrowright} \ar[d]^g &{I_1^X}' \ar@{}[dr]|{\circlearrowright} \ar[d] \ar[r]^{{i_1^X}'}  &\cdot\cdot\cdot \ar@{}[dr]|{\circlearrowright} \ar[r]^{{i_{n-1}^X}'} &{I_n^X}' \ar@{}[dr]|{\circlearrowright} \ar[d] \ar[r]^{{i_n^X}'}  &S'X \ar[d]^{S'g} \ar@{-->}[r]^-{(\del^X)'} &\\
Y \ar[r]_{{i_0^Y}'}  &{I_1^Y}'  \ar[r]_{{i_1^Y}'}  &\cdot\cdot\cdot \ar[r]_{{i_{n-1}^Y}'} &{I_n^Y}'  \ar[r]_{{i_n^Y}'}  &S'Y  \ar@{-->}[r]^-{(\del^Y)'} &
}
$$
and
$$\xymatrix{
X \ar[r]^{i_0^X} \ar@{}[dr]|{\circlearrowright} \ar[d]^g &I_1^X \ar@{}[dr]|{\circlearrowright} \ar[r]^{i_1^X} \ar[d] &\cdot\cdot\cdot \ar@{}[dr]|{\circlearrowright} \ar[r]^{i_{n-1}^X} &I_n^X \ar[d] \ar@{}[dr]|{\circlearrowright} \ar@{->>}[r]^{i_n^X} &SX \ar[d]^{Sg} \ar@{-->}[r]^-{\del^X} &\\
Y \ar[r]^{i_0^Y} \ar@{}[dr]|{\circlearrowright} \ar@{=}[d] &I_1^Y \ar@{}[dr]|{\circlearrowright} \ar[d] \ar[r]^{i_1^Y} &\cdot\cdot\cdot \ar@{}[dr]|{\circlearrowright} \ar[r]^{i_{n-1}^Y} &I_n^Y \ar@{}[dr]|{\circlearrowright} \ar[d] \ar[r]^{i_n^Y}  &SY \ar[d]^{f^Y} \ar@{-->}[r]^-{\del^Y} &\\
Y \ar[r]_{{i_0^Y}'}  &{I_1^Y}'  \ar[r]_{{i_1^Y}'}  &\cdot\cdot\cdot \ar[r]_{{i_{n-1}^Y}'} &{I_n^Y}'  \ar[r]_{{i_n^Y}'}  &S'Y  \ar@{-->}[r]^-{(\del^Y)'} &.
}
$$
By Lemma \ref{iso}, we have $\overline {S'gf^X}=\overline {f^YSg}$, hence we have the following commutative diagram
$$\xymatrix{
SX \ar[r]^{\Phi^X}_{\simeq} \ar@{}[dr]|{\circlearrowright} \ar[d]_{\overline {Sg}} &S'X \ar[d]^{\overline{S'g}}\\
SY \ar[r]_{\Phi^Y}^{\simeq} &S'Y.
}
$$
\end{proof}

Dually, for any object $Y$, we fix a distinguished  $n$-exangle $$\xymatrix{TY \ar[r]^{i_0^Y} &I_1^Y \ar[r]^{i_1^Y} \ar[r] &\cdot\cdot\cdot \ar[r]^{i_{n-1}^Y} &I_n^Y \ar[r]^{i_n^Y} \ar[r] &Y \ar@{-->}[r]^-{\rho^Y} &}$$ where $I_k^Y\in \I$ for $k\in \{1,2,...,n \}$, $i_0^Y$ is a left $\I$-appromation of $TY$ and $i_n^Y$ is a right $\I$-appromation of $Y$.

For a morphism $g\colon Y\to X$, there exists a commutative diagram
$$\xymatrix{
TY \ar[r]^{i_0^Y} \ar[d]_{Tg} \ar@{}[dr]|{\circlearrowright} &I_1^Y \ar@{}[dr]|{\circlearrowright} \ar[r]^{i_1^Y}  \ar[d] &\cdot\cdot\cdot \ar@{}[dr]|{\circlearrowright} \ar[r]^{i_{n-1}^Y} &I_n^Y \ar@{}[dr]|{\circlearrowright} \ar[d] \ar[r]^{i^n_Y} &Y \ar[d]^g \\
TX \ar[r]^{i_0^X}  &I_1^X \ar[r]^{i_1^X}  &\cdot\cdot\cdot \ar[r]^{i_{n-1}^X} &I_n^X \ar[r]^{i_n^X}  &X
}
$$
where  the morphism $\overline {Tg}$ is unique in $\C$ by the dual of Lemma \ref{iso}.

Hence we have a functor $T$ on $\overline \C$ such that $T\colon Y\to TY$ and $T\colon\overline g\to \overline {Tg}$.

By Lemma \ref{lemma}, we know that $S$ is a well-defined functor.
One can easily check that $T$ gives a quasi-inverse of $S$.
In summary, we have the following.

\begin{proposition}
The functor $S\colon\overline \C\to \overline \C$ is an autoequivalence and the functor $T$ is its  quasi-inverse.
\end{proposition}

\begin{lemma}\label{standard}
If we have a distinguished $n$-exangle $$\xymatrix{X_0 \ar[r]^{d_0^X} &X_1 \ar[r]^{d_1^X} \ar[r] &\cdot\cdot\cdot \ar[r]^{d_{n-1}^X} &X_{n} \ar[r]^{d_{n}^X} &X_{n+1} \ar@{-->}[r]^-{\sigma^{X_0}} &}$$
then we have the following commutative diagram
$$\xymatrix{X_0 \ar[r]^{d_0^X} \ar@{}[dr]|{\circlearrowright} \ar@{=}[d] &X_1 \ar@{}[dr]|{\circlearrowright} \ar[r]^{d_1^X} \ar[d]^{f_1} &\cdot\cdot\cdot \ar@{}[dr]|{\circlearrowright} \ar[r]^{d_{n-1}^X} &X_{n} \ar@{}[dr]|{\circlearrowright} \ar[r]^{d_{n}^X} \ar[d]^{f_n} &X_{n+1} \ar[d]^{d_{n+1}^X} \ar@{-->}[r]^{\sigma^{X_0}} &\\
X_0 \ar[r]_{i_0^{X^0}} &I_1^{X^0} \ar[r]_{i_1^{X^0}} \ar[r] &\cdot\cdot\cdot \ar[r]_{i_{n-1}^{X^0}} &I_n^{X^0} \ar[r]_{i_n^{X^0}} \ar[r] &SX_0 \ar@{-->}[r]^-{\delta^{X_0}} &
}
$$
which induces a distinguished $n$-exangle
$$\xymatrix{X_1 \ar[r]^-{\alpha_1} &X_2\oplus I_1^{X_0} \ar[r]^-{\alpha_2} &\cdot\cdot\cdot \ar[r]^-{\alpha_n} &X_{n+1}\oplus I_n^{X^0} \ar[r]^-{\alpha_{n+1}} \ar[r] &SX_0\ar@{-->}[rr]^-{(d_0^X)_*\del^{X_0}} &&,}$$
where $\alpha_1=\begin{bmatrix}-d^X_1\\ f_1\end{bmatrix},$
$\alpha_j=\begin{bmatrix}-d^X_{j}&0\\ f_{j}&i^{X^0}_{j-1}\end{bmatrix}\ (2\leq j\leq n),$
$\alpha_{n+1}=\begin{bmatrix}d_{n+1}^{X}&i^{X_0}_n\end{bmatrix}$.
\end{lemma}

\proof This follows from the definition of injective objects and (EA2). \qed
\medskip

For $\overline \C$, let $\overline \E: \overline {\C}^{\op} \times \overline \C\to \Ab$ and $\overline {\mathfrak{s}}$ be the bifunctor given by
\begin{itemize}
\item $\overline \E(C,A)=\E(C,A)$, $\forall A,C\in \C$.
\item $\overline \E(\overline c, \overline a)=\E(c,a)$, $\forall a\in \Hom_{\C}(A,A'),c\in \Hom_{\C}(C,C')$.
\item For any $\overline \E$-extension $\delta\in \overline \E(C,A)=\E(C,A)$, define
$$\overline {\mathfrak{s}}(\delta)=\overline {\mathfrak{s}(\delta)}=[\xymatrix{A \ar[r]^{\overline {d_0^X}} &X_1 \ar[r]^{\overline {d_1^X}} \ar[r] &\cdot\cdot\cdot \ar[r]^{\overline {d_{n-1}^X}} &X_{n} \ar[r]^{\overline {d_{n}^X}} &C \ar@{-->}[r]^-{\delta} &].}$$
\end{itemize}

Under this setting, any extension $\delta\in \E(C,A)$ indunces natural transformations

$$\overline {\delta_{\sharp}}:\overline {\C}(-,C)\to \E(-,A) \text{ and }\overline {\delta^{\sharp}}:\overline {\C}(A,-)\to \E(C,-).$$

For any $X\in \C$, we have
\begin{itemize}
\item $(\overline {\delta_{\sharp}})_{X}:\overline {\C}(X,C)\to \E(X,A): \overline f\mapsto f^*\delta$.
\item $\overline {\delta^{\sharp}_{X}}:\overline {\C}(A,X)\to \E(C,X): \overline g\mapsto g_*\delta$
\end{itemize}

\begin{lemma}\label{conflation}
Any morphism $d_{n+1}^X\colon X_{n+1}\to X_0$ admits the following commutative diagram
$$\xymatrix{TX
_0 \ar[r]^{d_0^X} \ar@{}[dr]|{\circlearrowright} \ar@{=}[d] &X_1 \ar@{}[dr]|{\circlearrowright} \ar[r]^{d_1^X} \ar[d]^{f_1} &\cdot\cdot\cdot \ar@{}[dr]|{\circlearrowright} \ar[r]^{d_{n-1}^X} &X_{n} \ar@{}[dr]|{\circlearrowright} \ar[r]^{d_{n}^X} \ar[d]^{f_n} &X_{n+1} \ar[d]^{d_{n+1}^X} \ar@{-->}[r]^-{(d_{n+1}^X)^* \delta} &\\
TX_0 \ar[r]_{i_0^{X^0}} &I_1^{X_0} \ar[r]_{i_1^{X_0}} \ar[r] &\cdot\cdot\cdot \ar[r]_{i_{n-1}^{X_0}} &I_n^{X_0} \ar[r]_{i_n^{X_0}} \ar[r] &X_0 \ar@{-->}[r]^-{\delta} &
}
$$
which induces an $(n+2)$-angle
$$\xymatrix{X_1 \ar[r]^-{\alpha_1} &X_2\oplus I_1^{X_0} \ar[r]^-{\alpha_2} &\cdot\cdot\cdot \ar[r]^-{\alpha_n} &X_{n+1}\oplus I_n^{X_0} \ar[r]^-{\alpha_{n+1}} \ar[r] &X_0\ar@{-->}[r]^-{(d_0^X)_*\del} &&}$$
where $\alpha_1=\begin{bmatrix}-d^X_1\\ f_1\end{bmatrix},$
$\alpha_j=\begin{bmatrix}-d^X_{j}&0\\ f_{j}&i^{X^0}_{j-1}\end{bmatrix}\ (2\leq j\leq n),$
$\alpha_{n+1}=\begin{bmatrix}d_{n+1}^{X}&i^{X_0}_n\end{bmatrix}$.
\end{lemma}

\begin{remark}\label{rem}
Note that $\overline{d_{n+1}^X}=\overline{\alpha_{n+1}}$. Thus Lemma \ref{conflation} implies that deflations are closed under composition in $\overline{\C}$.  Dually, the dual of Lemma \ref{conflation} implies the composition-closedness of deflations $\overline{\C}$.
\end{remark}

%To show this theorem, we need to show the following theorem first.

\begin{lemma}\label{naturaliso}
We have $ \overline {\C}(X,S(-))\simeq \E(X,-)$ and $ \overline {\C}(-,SX) \simeq \E(-,X)$.
\end{lemma}

\begin{proof}
It is not hard to check that the following diagrams
$$\xymatrix{
\overline {\C}(X,SX_0) \ar[r]^-{\overline {\delta^{\sharp}_{X}}}_\simeq \ar[d]_{\overline \C(X,Sf)} &\E(X,X_0) \ar[d]^{f_*}\\
\overline {\C}(X,SY_0)\ar[r]^-{\overline {\delta^{\sharp}_{Y}}}_\simeq &\E(X,Y_0)\\
} \quad
\xymatrix{\overline {\C}(Y_0,SX)\ar[r]^-{(\overline {\delta_{\sharp}})_{Y_0}}_\simeq \ar[d]_{\overline \C(f,SX)} &\E(Y_0,X) \ar[d]^{f^*}\\
\overline {\C}(X_0,SX)\ar[r]^-{(\overline {\delta_{\sharp}})_{X_0}}_\simeq  &\E(X_0,X)
}
$$
commute for any $X_0, Y_0\in \C$ by Lemma \ref{iso}.
\end{proof}

\begin{lemma} Let $(\C,\E,\mathfrak{s})$ be a Frobenius $n$-exangulated category. Then
$(\overline {\C},\overline{\E},\overline{\mathfrak{s}})$ is an $n$-exangulated category.
\end{lemma}

\begin{proof}
(I) We first check that if
$$\xymatrix{A \ar[r]^{\overline {d_0^X}} &X_1 \ar[r]^{\overline {d_1^X}} \ar[r] &\cdot\cdot\cdot \ar[r]^{\overline {d_{n-1}^X}} &X_{n} \ar[r]^{\overline {d_{n}^X}} &C \ar@{-->}[r]^-{\delta} &}$$
is the image of a distinguished $n$-exangle in $(\C,\E,\mathfrak{s})$, then itself is a distinguished $n$-exangle in $(\overline \C,\overline {\E},\overline {\mathfrak{s}})$.
We check  that the following sequence
$$\overline {\C}(C,-) \xrightarrow{\overline {\C}(\overline {d_n^X},-)} \cdot\cdot\cdot \xrightarrow{\overline {\C}(\overline {d_0^X},-)}\overline {\C}(A,-)\xrightarrow{\overline {\delta^{\sharp}}} \overline {\E}(C,-) $$
is exact.

If we have a morphism $f:A\to X$ such that $\overline {\delta^{\sharp}_{X}}(\overline f)=f_*\delta$. Then $f$ factors through $d_n^X$. If we have $\overline f$ factors through $\overline {d_{0}^X}$, then there is a diagram
$$\xymatrix{
A \ar[r]^{d_0^X} \ar[dr]^{f} \ar[d]_{i} &X_1 \ar[d]^{g}\\
I \ar[r]_{i'} &X
}
$$
where $I\in \mathcal I$ and $f-gd_0^X=i'i$. But $i$ also factors through $d_0^X$, which implies $f$ factors through $d_0^X$, hence $\overline {\delta^{\sharp}_{X}}(\overline f)=0=f_*\delta$.

If we have a morphism $f_1\colon X_1\to X$ such that $\overline {f_1d_0^X}=0$, then there is a commutative diagram
$$\xymatrix{
A \ar[r]^{d_0^X} \ar[d]_{i_1} &X_1 \ar[d]^{f_1}\\
I_1 \ar[r]_{i_1'} &X
}
$$
where $I_1\in \mathcal I$. Since $i_1$ factors through $d_0^X$, we get $\overline {f_1}$ factors through $\overline {d_{1}^X}$.

If $\overline {f_1}$ factors through $\overline {d_1^X}$, since $\overline {d_1^Xd_0^X}=0$, we have $\overline {f_1d_0^X}=0$.
By Lemma \ref{standard}, we get a distinguished $n$-exangle
$$\xymatrix{X_1 \ar[r]^{-\overline {d_1^X}} \ar[r] &\cdot\cdot\cdot \ar[r]^{-\overline {d_{n-1}^X}} &X_{n} \ar[r]^{-\overline {d_{n}^X}} &C \ar[r]^{\overline {d_{n+1}^X}} &SA \ar@{-->}[r]^-{(d_0^X)_*\delta} &}$$
hence we can get the exactness of $\overline {\C}(X_3,-) \xrightarrow{\overline {\C}(\overline {d_2^X},-)}\overline {\C}(X_2,-)\xrightarrow{\overline {\C}(\overline {d_1^X},-)} \overline {\C}(X_1,-)$ by the similar argument as above.

The distinguished $n$-exangles in $(\overline {\C},\overline{\E},\overline{\mathfrak{s}})$  are the image of distinguished $n$-exangles in $(\C,\E,\mathfrak{s})$.

(II) We check the axioms of $n$-exangulated category.

(EA1) This follows from Lemma \ref{conflation} and Remark \ref{rem}.

(EA2) For $\rho\in \overline {\E}(D,A)$ and $\overline c\in \overline {\C}(C,D)$, assume that ${_A}\langle \overline X^{\centerdot}, \overline c^*\rho \rangle_C$ and ${_A}\langle \overline Y^{\centerdot}, \rho \rangle_D$ are two distinguished $n$-exangles in $(\overline {\C},\overline{\E},\overline{\mathfrak{s}})$, then we have a good lift of $(1_A, c)$: $f^{\centerdot}$ in $(\C,\E,\mathfrak{s})$
$$\xymatrix{
A \ar[r]^{x_0} \ar@{}[dr]|{\circlearrowright} \ar@{=}[d] &X_1 \ar[r]^{x_1} \ar@{}[dr]|{\circlearrowright} \ar[d] &\cdot\cdot\cdot \ar[r]^{x_{n-1}} \ar@{}[dr]|{\circlearrowright} &X_n \ar[d] \ar@{}[dr]|{\circlearrowright} \ar[r]^{x_n} &C \ar[d]^{c} \ar@{-->}[r]^-{c^*\rho} &\\
A \ar[r]_{y_0} &Y_1 \ar[r]_{y_1} &\cdot\cdot\cdot \ar[r]_{y_{n-1}} &Y_n \ar[r]_{y_n}  &D \ar@{-->}[r]^-{\rho} &
}
$$
which gives a distinguished $n$-exangle $\langle M^{\centerdot}_f, (d_0^X)_*\rho \rangle$. Hence $(\overline{1_A}, \overline{c})$ has a good lift $\overline {f^{\centerdot}}$ in $(\overline {\C},\overline{\E},\overline{\mathfrak{s}})$ giving a distinguished $n$-exangle $\langle \overline M^{\centerdot}_f, (\overline {d^0_X})_*\rho \rangle$.
\end{proof}

By the previous lemmas, we may let $\E_S(-,-)=\overline \C(-,S(-))\simeq \E(-,-)$. For any distinguished $n$-exangle
$$\xymatrix{A \ar[r]^{\overline {d_0^X}} &X_1 \ar[r]^{\overline {d_1^X}} \ar[r] &\cdot\cdot\cdot \ar[r]^{\overline {d_{n-1}^X}} &X_{n} \ar[r]^{\overline {d_{n}^X}} &C \ar@{-->}[r]^-{\delta} &}$$
in $\overline \C$, since there is a one-to-one correspondence between $\E(C,A)$ and $\overline {\C}(C,SA)$, it induces a
$$\xymatrix{A \ar[r]^{\overline {d_0^X}} &X^1 \ar[r]^{\overline {d_1^X}} \ar[r] &\cdot\cdot\cdot \ar[r]^{\overline {d_{n-1}^X}} &X_{n} \ar[r]^{\overline {d_{n}^X}} &C \ar[r]^-{\overline {d_{n+1}^X}} &SA}$$
We call such sequence an $(n+2)$-$S$-sequence. It is not hard to check the exactness of such sequence. Let $\square_{\overline{\mathfrak{s}}}$ be the class of $(n+2)$-$S$-sequences. Then by \cite[Proposition 4.8]{HLN}, we have the following our first main result:

\begin{theorem}\label{main1} Let $(\C,\E,\mathfrak{s})$ be a Frobenius $n$-exangulated category.
Then $(\overline \C,S, \square_{\overline{\mathfrak{s}}})$ is an $(n+2)$-angulated category.
\end{theorem}

Since any Frobenius $n$-exact category can be viewed as a Frobenius $n$-exangulated category, then
this generalizes a result by Jasso \cite[Theorem 5.11]{J}.
After completing this work, we found that Zheng and Wei \cite{ZW} gave another proof method about this result.

\section{Some new Frobenius $n$-exangulated categories}

Let $(\C,\E,\s)$ be an $n$-exangulated category and $\F\subseteq\E$ an additive subfunctor.
A distinguished $n$-exangle
$$A_0\xrightarrow{}A_1\xrightarrow{}A_2\xrightarrow{}\cdots\xrightarrow{}A_{n-1}
\xrightarrow{}A_n\xrightarrow{}A_{n+1}\overset{\delta}{\dashrightarrow}$$
is said to be $\F$-\emph{exangle} if $\del$ is in $\F(C,A)$.

When constructing subfunctors $\F$ of $\E\colon\C\op\times\C\to\Ab$, one must in particular show that
$\F(C,A)$ is a subgroup of $\E(C,A)$ for all pairs of objects $A$ and $C$ in $\C$.
The following result is useful to show this property.

\begin{lemma}\label{lem1}
Let $\F$ be a subfunctor of $\E\colon\C\op\times\C\to\sets$. Then
$\F$ is an additive subfunctor of $\E\colon\C\op\times\C\to\Ab$
if and only if $\F$ is closed under direct sums of $\F$-exangles.
\end{lemma}

\proof This proof is an adaptation of the proof of \cite[Lemma 1.1]{AS}.
We sketch the proof for the convenience of the reader.

Let $\F$ be a subfunctor of $\E\colon\C\op\times\C\to\sets$.
Then for all objects $A$ and $C$ in $\C$, we have that $0$ is in $\F(C,A)$ and if
$\del$ is in $\F(C,A)$, then $-\del$ is in $\F(C,A)$.
The proof of these facts goes as follows.
Let
$$A\xrightarrow{}A_1\xrightarrow{}A_2\xrightarrow{}\cdots\xrightarrow{}A_{n-1}
\xrightarrow{}A_n\xrightarrow{}C\overset{\delta}{\dashrightarrow}$$
be  an $\F$-\emph{exangle}. Since $\E$ is an additive bifunctor, the zero element in
$\E(C,A)$ is $\E(0,\textrm{id}_A)(\del)$ and the inverse of $\del$ is
$\E(-\textrm{id}_C,\textrm{id}_A)(\del)$. Since $\F$ is a subfunctor of $\E$
and $\del$ is $\F$-exangle, the sequences $\E(0,\textrm{id}_A)(\del)$
and $\E(-\textrm{id}_C,\textrm{id}_A)(\del)$ are again $\F$-exangle.
Therefore, $\F(C,A)$ is a subgroup of $\E(C,A)$ for all pairs of objects
$A$ and $C$ in $\C$ if and only if $\F(C,A)$ is closed under Baer sum.

Assume that $\F$ is closed under direct sums of $\F$-exangles.
We want to show that $\F$ is an additive subfunctor $\E\colon\C\op\times\C\to\Ab$.
It suffices to show that $\F(C,A)$ is a subgroup of $\E(C,A)$ for a pair of objects
$A$ and $C$ in $\C$ and that $\F$ is an additive bifunctor.
Let
$$A\xrightarrow{}A_1\xrightarrow{}A_2\xrightarrow{}\cdots\xrightarrow{}A_{n-1}
\xrightarrow{}A_n\xrightarrow{}C\overset{\delta}{\dashrightarrow}$$
and
$$A\xrightarrow{}B_1\xrightarrow{}B_2\xrightarrow{}\cdots\xrightarrow{}B_{n-1}
\xrightarrow{}B_n\xrightarrow{}C\overset{\delta'}{\dashrightarrow}$$
be in $\F(C,A)$. Then the Baer sum $\del+\del'$ is obtained by $\del+\del'=\E(\Delta_C,\nabla_A)(\del\oplus\del'),$
where $\Delta_C=\left[
              \begin{smallmatrix}
                1\\ 1
              \end{smallmatrix}
            \right]\colon C\to C\oplus C$ and $\nabla_A=\left[
              \begin{smallmatrix}
                1&1
              \end{smallmatrix}
            \right]\colon A\oplus A\to A$.
Since $\del\oplus\del'$ is in $\F(C\oplus C,A\oplus A)$ and
$\F$ is a subfunctor of $\E$ by assumption, it follows that $\del+\del'$ is in $\F(C,A)$ and
$\F(C,A)$ is a subgroup of $\E(C,A)$.

The remaining assertion is similar to \cite[Lemma 1.1]{AS}.  \qed
\medskip

Let $(\C,\E,\s)$ be an $n$-exangulated category and $\X$ be a subcategory of $\C$.
For each pair of objects $A$ and $C$ in $\C$, define
$$\F^{\X}(C,A)=
\{A_0\xrightarrow{f}A_1\xrightarrow{}A_2\xrightarrow{}\cdots\xrightarrow{}A_{n-1}
\xrightarrow{}A_n\xrightarrow{}A_{n+1}\overset{\delta}{\dashrightarrow}
\,\mid\, f\ \textrm{is an}\ \X\textrm{-monic}\}.$$
Dually, we define for each pair of objects $A$ and $C$ in $\C$
$$\F_{\X}(C,A)=
\{A_0\xrightarrow{}A_1\xrightarrow{}A_2\xrightarrow{}\cdots\xrightarrow{}A_{n-1}
\xrightarrow{}A_n\xrightarrow{g}A_{n+1}\overset{\delta}{\dashrightarrow}
\,\mid\, g\ \textrm{is an}\ \X\textrm{-epic}\}.$$

It is not obvious that these constructions give additive subfunctors of
$\E$, a fact we prove next.

\begin{lemma}\label{lem2}
Let $(\C,\E,\s)$ be an $n$-exangulated category and $\X$ be any subcategory of $\C$. Then
$\F^{\X}$ and $\F_{\X}$ are additive subfunctor of $\E$.
\end{lemma}

\proof We only show that $\F^{\X}$ is an additive subfunctor of $\E$, since
the proof for  $\F_{\X}$ is given by the dual argument.

By Lemma \ref{lem1}, it is enough to show that $\F^{\X}$ is closed under
direct sums of $\F^{\X}$-exangles and  that $\F^{\X}$ is a subfunctor of $\E\colon\C\op\times\C\to\Ab$.

We first show that $\F^{\X}$ is a subfunctor. Let
$$A\xrightarrow{\alpha_0}A_1\xrightarrow{\alpha_1}A_2\xrightarrow{\alpha_2}\cdots\xrightarrow{\alpha_{n-2}}A_{n-1}
\xrightarrow{\alpha_{n-1}}A_n\xrightarrow{\alpha_n}C\overset{\delta}{\dashrightarrow}$$
be in $\F^{\X}(C,A)$ and let $u\colon A\to B$ be a morphism in $\C$. Then
there exists a morphism of distinguished $n$-exangles
$$\xymatrix{A\ar[r]^{\alpha_0}\ar[d]^u&A_1\ar[r]^{\alpha_1}\ar[d]^{\omega_1}&A_2\ar[r]^{\alpha_2}\ar[d]^{\omega_2}&A_3\ar[r]^{\alpha_3}\ar[d]^{\omega_3}
&\cdot\cdot\cdot\ar[r]^{\alpha_{n-2}}&A_{n-1}\ar[r]^{\alpha_{n-1}}
\ar[d]^{\omega_{n-1}}&A_n\ar[r]^{\alpha_n}\ar[d]^{\omega_{n}}&C\ar@{-->}[r]^{\del}\ar@{=}[d]&\\
B\ar[r]^{\beta_0}&B_1\ar[r]^{\beta_1}&B_2\ar[r]^{\beta_2}&B_3\ar[r]^{\beta_3}&\cdot\cdot\cdot\ar[r]^{\beta_{n-2}}&B_{n-1}\ar[r]^{\beta_{n-1}}&B_n\ar[r]^{\beta_n}& C\ar@{-->}[r]^{a_{\ast}\delta}&}$$
Thus its mapping cocone
$$A\xrightarrow{\left[
              \begin{smallmatrix}
                -\alpha_0\\ u
              \end{smallmatrix}
            \right]}A_1\oplus B\xrightarrow{\left[
              \begin{smallmatrix}
                -\alpha_1&0\\ \omega_1&\beta_0
              \end{smallmatrix}
            \right]}A_2\oplus B_1\xrightarrow{\left[
              \begin{smallmatrix}
                -\alpha_2&0\\ \omega_3&\beta_1
              \end{smallmatrix}
            \right]}\cdots\xrightarrow{\left[
              \begin{smallmatrix}
                -\alpha_{n-1}&0\\ \omega_{n-1}&\beta_{n-2}
              \end{smallmatrix}
            \right]}A_{n}\oplus B_{n-1}\xrightarrow{\left[
              \begin{smallmatrix}
                \omega_n&\beta_{n-1}
              \end{smallmatrix}
            \right]}B_n\overset{~~}{\dashrightarrow}$$
is a distinguished $n$-exangle.

We claim that $\beta_0$ is an $\X$-monic.  Indeed, let $a\colon B\to X$ be any
morphism with $X\in\X$. Since $\alpha_0$ is an $\X$-monic, there exists a morphism
$b\colon B\to X$ such that $b\alpha_0=au$. It follows that
$\left[\begin{smallmatrix}
                b&a
              \end{smallmatrix}
            \right]\left[
              \begin{smallmatrix}
                -\alpha_0\\ u
              \end{smallmatrix}
            \right]=0$.  Thus there exists a morphism
$\left[\begin{smallmatrix}
                c&d
              \end{smallmatrix}
            \right]\colon A_2\oplus B_1\to X$ such that $$\left[
              \begin{smallmatrix}
               c&d
              \end{smallmatrix}
            \right]=\left[
              \begin{smallmatrix}
              b&a
              \end{smallmatrix}
            \right]\left[
              \begin{smallmatrix}
                -\alpha_1&0\\
                \omega_1&\beta_0
              \end{smallmatrix}
            \right].$$ In particular, we have $a=d\beta_0$.
This shows that $\beta_0$ is an $\X$-monic. Hence we obtain that $a_{\ast}\del$ is in $\F^{\X}(C,B)$.

Let $c\colon D\to C$ be any morphism in $\C$. Then
there is a morphism of distinguished $n$-exangles
$$\xymatrix{A\ar[r]^{\gamma_0}\ar@{=}[d]&D_1\ar[r]^{\gamma_1}\ar[d]^{\phi_1}&D_2\ar[r]^{\gamma_2}\ar[d]^{\phi_2}&D_3\ar[r]^{\gamma_3}\ar[d]^{\phi_3}
&\cdot\cdot\cdot\ar[r]^{\gamma_{n-2}}&D_{n-1}\ar[r]^{\gamma_{n-1}}
\ar[d]^{\phi_{n-1}}&D_n\ar[r]^{\gamma_n}\ar[d]^{\phi_{n}}&D\ar@{-->}[r]^{c^{\ast}\del}\ar[d]^{c}&\\
A\ar[r]^{\alpha_0}&A_1\ar[r]^{\alpha_1}&A_2\ar[r]^{\alpha_2}&A_3\ar[r]^{\alpha_3}&\cdot\cdot\cdot\ar[r]^{\alpha_{n-2}}&A_{n-1}\ar[r]^{\alpha_{n-1}}&A_n\ar[r]^{\alpha_n}& C\ar@{-->}[r]^{\delta}&}$$
We claim that $\gamma_0$ is an $\X$-monic. Indeed, let $i\colon A\to X'$ be any
morphism with $X'\in\X$. Since $\alpha_0$ is an $\X$-monic, there exists a morphism
$j\colon B\to X'$ such that $i=j\alpha_0$ and then $i=j\phi_1\gamma_0$. This shows that
$\gamma_0$ is an $\X$-monic. Thus we have that $c^{\ast}\del$ is in $\F^{\X}(C',A)$.
It is easy to check that $\F^{\X}\colon\C\op\times\C\to\sets$ is a bifunctor, hence
$\F^{\X}$ is a subfunctor of $\E\colon\C\op\times\C\to\sets$.

Since the subfunctor $\F^{\X}$ clearly is closed under direct sums of
$\F^{\X}$-exangle, the proof is complete.  \qed

\begin{lemma}\label{key}
Let $(\C,\E,\s)$ be an $n$-exangulated category and $\X$ be any subcategory of $\C$. Then
\begin{itemize}
\item[\rm (1)] $(\C,\F^{\X},\s_{\F^{\X}})$ is an $n$-exangulated category.
\item[\rm (2)] $(\C,\F_{\X},\s_{\F_{\X}})$ is an $n$-exangulated category.
\end{itemize}
\end{lemma}

\proof  By Lemma \ref{lem2} and Lemma \ref{a6}, $\X$-monomorphisms are closed under composition, and $\X$-epimorphism are closed under composition. The result then follows.  \qed

\begin{theorem}\label{main2}
Let $(\C,\E,\s)$ be an $n$-exangulated category.
If $\X$ is a strongly functorially finite subcategory of $\C$, then $(\C,\F,\s_{\F})$ is a Frobenius $n$-exangulated category whose projective-injective objects are precisely $\X$, where $\F:=\F^{\X}\cap~\F_{\X}$.
\end{theorem}

\proof By Lemma \ref{key}, we have that $(\C,\F,\s_{\F})$ is an $n$-exangulated category.
By the definition of $\F$, we know that every object in $\X$ is projective and injective in $(\C,\F,\s_{\F})$.

Take any projective object $P$ in $(\C,\F,\s_{\F})$, since $\X$ is a strongly contravariantly finite subcategory of $\C$, there exists a distinguished  $n$-exangle
$$C\xrightarrow{}X_1\xrightarrow{}X_2\xrightarrow{}\cdots\xrightarrow{}X_{n-1}\xrightarrow{}X_{n}\xrightarrow{~g~}P\overset{}{\dashrightarrow}$$
where $g$ is a right $\X$-approximation of $P$ and $X_i\in\X$.
By Lemma \ref{a5}, we have that this $n$-exangle splits implying that $P\in\X$. The injective objects can be treated similarly.
Hence the classes of projective objects and injective objects in $(\C,\F,\s_{\F})$ both equal $\X$.

Since $\X$ is a strongly functorially finite subcategory of $\C$, we have that
$(\C,\F,\s_{\F})$ has enough projectives and enough injectives.
This completes the proof.  \qed

\begin{remark}\label{rem1} In the case $\X\not= \{0\}$, the new Frobenius $n$-exangulated category  $(\C,\F,\s_{\F})$ in Theorem \ref{main2} is not $(n+2)$-angulated, since $\X$ is projective and injective and non-zero.  In the case $\X\not=\C$, it is easy to see that the new Frobenius $n$-exangulated category $(\C,\F,\s_{\F})$  in Theorem \ref{main2} is not $n$-exact. Otherwise any $\F$-extension splits and then any object in $\C$ is projective and injective. Then $\X=\C$, a contradiction. We take a subcategory $\{0\}\not=\X\varsubsetneqq \C$ satisfying the condition  in Theorem \ref{main2}, Then the new Frobenius $n$-exangulated category  $(\C,\F,\s_{\F})$  in Theorem \ref{main2} is neither $n$-exact nor $(n+2)$-angulated. This provides  us with many new $n$-exangulated categories which are neither $n$-exact nor $(n+2)$-angulated.
\end{remark}

\section{Examples}
In this section, we will give some examples to explain our main results.

\begin{example}
Let $\C$ be a Frobenius $n$-exact category. It is easy to see that $\X:=\cal P=\cal I$ is  a strongly functorially finite subcategory of $\C$. By Theorem \ref{main2}, we have that $(\C,\F,\s_{\F})$ is a Frobenius $n$-exangulated category whose projective-injective objects are precisely $\X$, where $\F:=\F^{\X}\cap~\F_{\X}$. By Theorem \ref{main1}, we know that $\C/\X$ is an $(n+2)$-angulated category.
\end{example}

Let $\C$ be $(n+2)$-angulated category. From this point on, when we say that $\X$ is
a subcategory of $\C$, we always mean that $\X$ is full and closed under isomorphisms,
direct sums and direct summands.

The notion of mutation pairs of subcategories in an $(n+2)$-angulated category was
defined by Lin \cite[Definition 3.1]{L}. We recall the definition here.

\begin{definition}
Let $\C$ be an $(n+2)$-angulated category, and $\X\subseteq\A$ be two subcategories of $\C$.
 The pair $(\A,\A)$ is called a $\X$-\emph{mutation pair} if it satisfies the following conditions:
\begin{itemize}
\item[(1)] For any object $A\in\A$, there exists an $(n+2)$-angle $$A\xrightarrow{x_0}X_1\xrightarrow{x_1}X_2\xrightarrow{x_2}\cdots\xrightarrow{x_{n-2}}X_{n-1}\xrightarrow{x_{n-1}}X_{n}\xrightarrow{x_{n}}B\xrightarrow{x_{n+1}}\Sigma A$$
where $X_i\in\X, B\in\A, x_0$ is a left $\X$-approximation of $A$ and $x_{n}$ is a right $\X$-approximation of $B$.

\item[(2)] For any object $C\in\A$, there exists an $(n+2)$-angle $$D\xrightarrow{x'_0}X'_1\xrightarrow{x'_1}X'_2\xrightarrow{x'_2}\cdots\xrightarrow{x'_{n-2}}X'_{n-1}\xrightarrow{x'_{n-1}}X'_n\xrightarrow{x'_{n}}C\xrightarrow{x'_{n+1}}\Sigma D$$
where $X'_i\in\X, D\in\A, x'_0$ is a left $\X$-approximation of $D$ and $x'_{n}$ is a right $\X$-approximation of $C$.
\end{itemize}
\end{definition}

\begin{lemma}\label{b1}
Let $\C$ be an $(n+2)$-angulated category and $\X\subseteq\A$ two subcategories of $\C$.
Assume that $$A_0\xrightarrow{\alpha_0}A_1\xrightarrow{\alpha_1}A_2\xrightarrow{\alpha_2}\cdots\xrightarrow{\alpha_{n-2}}A_{n-1}
\xrightarrow{\alpha_{n-1}}A_n\xrightarrow{\alpha_n}A_{n+1}\xrightarrow{\alpha_{n+1}}\Sigma A_0.$$
is any $(n+2)$-angle in $\A$.
If $(\A,\A)$ is an $\X$-mutation pair, then $\alpha_0$ is an $\X$-monic if and only if $\alpha_n$ is an $\X$-epic.
\end{lemma}

\proof Suppose that $\alpha_0$ is an $\X$-monic.
For object $A_0\in\A$, since$(\A,\A)$ is an $\X$-mutation pair, then there exists an $(n+2)$-angle
$$A_0\xrightarrow{f}X_1\xrightarrow{x_1}X_2\xrightarrow{x_2}\cdots\xrightarrow{x_{n-2}}X_{n-1}\xrightarrow{x_{n-1}}X_{n}\xrightarrow{g}B_0\xrightarrow{h}\Sigma A_0$$
where $f$ is a left $\X$-approximation of $A_0$ and $g$ is a right $\X$-approximation of $B_0$.
Consider the following diagram
$$\xymatrix{A_0\ar[r]^{\alpha_0}\ar@{=}[d]&A_1\ar[r]^{\alpha_1}\ar[d]^{\omega_1}&A_2\ar[r]^{\alpha_2}\ar@{-->}[d]^{\omega_2}&A_3\ar[r]^{\alpha_3}\ar@{-->}[d]^{\omega_3}
&\cdot\cdot\cdot\ar[r]^{\alpha_{n-2}}&A_{n-1}\ar[r]^{\alpha_{n-1}}\ar@{-->}[d]^{\omega_{n-1}}&A_n\ar[r]^{\alpha_n}\ar@{-->}[d]^{\omega_{n}}&A_{n+1}\ar[r]^{\alpha_{n+1}}\ar@{-->}[d]^{\omega_{n+1}}&\Sigma A_0\ar@{=}[d]\\
A_0\ar[r]^{f}&X_1\ar[r]^{x_1}&X_2\ar[r]^{x_2}&X_3\ar[r]^{x_3}&\cdot\cdot\cdot\ar[r]^{x_{n-2}}&X_{n-1}\ar[r]^{x_{n-1}}&X_n\ar[r]^{g}& B_0\ar[r]^{h}&\Sigma A_0}$$
where $\omega_1$ exists since $\alpha_0$ is an $\X$-monic.
By \cite[Lemma 4.1]{BT}, there are  morphisms $\omega_2,\omega_3,\cdots,\omega_n$ in $\C$ which make the above diagram commutative and there exists $(n+2)$-angle
$$A_1\xrightarrow{\left[
              \begin{smallmatrix}
                -\alpha_1\\ \omega_1
              \end{smallmatrix}
            \right]}A_2\oplus X_1\xrightarrow{\left[
              \begin{smallmatrix}
                \alpha_2&0\\ \omega_2&x_1
              \end{smallmatrix}
            \right]}\cdots\xrightarrow{\left[
              \begin{smallmatrix}
                \alpha_{n}&0\\ (-1)^n\omega_{n}&x_{n-1}
              \end{smallmatrix}
            \right]}A_{n+1}\oplus X_{n}\xrightarrow{\left[
              \begin{smallmatrix}
                (-1)^{n+1}\omega_{n+1}&g
              \end{smallmatrix}
            \right]}B_0\xrightarrow{\Sigma\alpha_0\circ h}\Sigma A_1.$$

Now we show that $\alpha_{n}$ is an $\X$-epic. In fact, for any morphism
$a\colon X\to A_{n+1}$ with $X\in\X$, there exists a morphism $b\colon X\to X_{n}$ such that
$gb=\omega_{n+1}a$ since $g$ is a right $\X$-approximation of $B_0$.
It follows that $$\left[
              \begin{smallmatrix}
                (-1)^{n+1}\omega_{n+1}&g
              \end{smallmatrix}
            \right]\left[
              \begin{smallmatrix}
                (-1)^na\\b
              \end{smallmatrix}
            \right]=0.$$
So there exists a morphism $\left[
              \begin{smallmatrix}
                c\\d
              \end{smallmatrix}
            \right]\colon X\to A_{n-1}\oplus X_{n-2}$ such that $$\left[
              \begin{smallmatrix}
                \alpha_{n-1}&0\\ (-1)^n\omega_{n-1}&x_{n-2}
              \end{smallmatrix}
            \right]\left[
              \begin{smallmatrix}
                c\\d
              \end{smallmatrix}
            \right]=\left[
              \begin{smallmatrix}
                (-1)^na\\b
              \end{smallmatrix}
            \right]$$
In particular, we have $\alpha_{n}c=(-1)^na$ and then $\alpha_{n}((-1)^nc)=a$.
This shows that $\alpha_{n}$ is an $\X$-epic.
The converse
implication is dual.   \qed
\medskip

Let $\C$ be an $(n+2)$-angulated category. Recall that a subcategory $\A$ of $\C$ is called
\emph{extension closed} if for any morphism $\alpha_{n+1}\colon A_{n+1}\to\Sigma A_0$ with
$A_0,A_{n+1}\in\A$, there exists an $(n+2)$-angle
$$A_0\xrightarrow{\alpha_0}A_1\xrightarrow{\alpha_1}A_2\xrightarrow{\alpha_2}\cdots\xrightarrow{\alpha_{n-2}}A_{n-1}
\xrightarrow{\alpha_{n-1}}A_n\xrightarrow{\alpha_n}A_{n+1}\xrightarrow{\alpha_{n+1}}\Sigma A_0.$$
where each $A_i\in\A$.

\begin{lemma}\label{b2}
Let $\C$ be an $(n+2)$-angulated category and $\X\subseteq\A$ be two subcategories of $\C$.
If $(\A,\A)$ is an $\X$-mutation pair and $\A$ is extension closed, then
\begin{itemize}
\item[\rm (1)] for any $\X$-monic $\alpha_0\colon A_0\to A_1$ in $\A$, there exists
an $(n+2)$-angle
$$A_0\xrightarrow{\alpha_0}A_1\xrightarrow{\alpha_1}A_2\xrightarrow{\alpha_2}\cdots\xrightarrow{\alpha_{n-2}}A_{n-1}
\xrightarrow{\alpha_{n-1}}A_n\xrightarrow{\alpha_n}A_{n+1}\xrightarrow{\alpha_{n+1}}\Sigma A_0$$
where each $A_i\in\A$.

\item[\rm (2)]for any $\X$-epic $\alpha_n\colon A_{n-1}\to A_n$, there exists
an $(n+2)$-angle
$$A_0\xrightarrow{\alpha_0}A_1\xrightarrow{\alpha_1}A_2\xrightarrow{\alpha_2}\cdots\xrightarrow{\alpha_{n-2}}A_{n-1}
\xrightarrow{\alpha_{n-1}}A_n\xrightarrow{\alpha_n}A_{n+1}\xrightarrow{\alpha_{n+1}}\Sigma A_0$$
where each $A_i\in\A$.
\end{itemize}
\end{lemma}

\proof We only give the proof of (1), the proof of (2) is obtained dually.

Suppose that $\alpha_0$ is an $\X$-monic in $\A$, we complete $\alpha_0$ to an $(n+2)$-angle
$$A_0\xrightarrow{\alpha_0}A_1\xrightarrow{\alpha_1}A_2\xrightarrow{\alpha_2}\cdots\xrightarrow{\alpha_{n-2}}A_{n-1}
\xrightarrow{\alpha_{n-1}}A_n\xrightarrow{\alpha_n}A_{n+1}\xrightarrow{\alpha_{n+1}}\Sigma A_0.$$
For object $A_0\in\A$, since$(\A,\A)$ is an $\X$-mutation pair, then there exists an $(n+2)$-angle
$$A_0\xrightarrow{f}X_1\xrightarrow{x_1}X_2\xrightarrow{x_2}\cdots\xrightarrow{x_{n-2}}X_{n-1}\xrightarrow{x_{n-1}}X_{n}\xrightarrow{g}B_0\xrightarrow{h}\Sigma A_0$$
where $B_0\in\A$, $f$ is a left $\X$-approximation of $A_0$ and $g$ is a right $\X$-approximation of $B_0$.
Consider the following diagram
$$\xymatrix{A_0\ar[r]^{\alpha_0}\ar@{=}[d]&A_1\ar[r]^{\alpha_1}\ar[d]^{\omega_1}&A_2\ar[r]^{\alpha_2}\ar@{-->}[d]^{\omega_2}&A_3\ar[r]^{\alpha_3}\ar@{-->}[d]^{\omega_3}
&\cdot\cdot\cdot\ar[r]^{\alpha_{n-2}}&A_{n-1}\ar[r]^{\alpha_{n-1}}\ar@{-->}[d]^{\omega_{n-1}}&A_n\ar[r]^{\alpha_n}\ar@{-->}[d]^{\omega_{n}}&A_{n+1}\ar[r]^{\alpha_{n+1}}\ar@{-->}[d]^{\omega_{n+1}}&\Sigma A_0\ar@{=}[d]\\
A_0\ar[r]^{f}&X_1\ar[r]^{x_1}&X_2\ar[r]^{x_2}&X_3\ar[r]^{x_3}&\cdot\cdot\cdot\ar[r]^{x_{n-2}}&X_{n-1}\ar[r]^{x_{n-1}}&X_n\ar[r]^{g}& B_0\ar[r]^{h}&\Sigma A_0}$$
where $\omega_1$ exists since $\alpha_0$ is an $\X$-monic.
By \cite[Lemma 4.1]{BT}, there are  morphisms $\omega_2,\omega_3,\cdots,\omega_n$ in $\C$ which make the above diagram commutative and there exists  an $(n+2)$-angle
$$A_1\xrightarrow{\left[
              \begin{smallmatrix}
                -\alpha_1\\ \omega_1
              \end{smallmatrix}
            \right]}A_2\oplus X_1\xrightarrow{\left[
              \begin{smallmatrix}
                \alpha_2&0\\ \omega_2&x_1
              \end{smallmatrix}
            \right]}\cdots\xrightarrow{\left[
              \begin{smallmatrix}
                \alpha_{n}&0\\ (-1)^n\omega_{n}&x_{n-1}
              \end{smallmatrix}
            \right]}A_{n+1}\oplus X_{n}\xrightarrow{\left[
              \begin{smallmatrix}
                (-1)^{n+1}\omega_{n+1}&g
              \end{smallmatrix}
            \right]}B_0\xrightarrow{\Sigma\alpha_0\circ h}\Sigma A_1.$$
Since $\A$ is extension closed and $B_0,A_1\in\A$, we have $A_2\oplus X_1,A_3\oplus X_2,\cdots,A_{n+1}\oplus X_{n}\in\A$ and then
$A_2,A_3,A_4,\cdots,A_{n+1}\in\A$.  \qed

\begin{corollary}\emph{\cite[Theorem 3.7]{L}}\label{cor1}
Let $\C$ be an $(n+2)$-angulated category and $\X\subseteq\A$ be two subcategories of $\C$.
If $(\A,\A)$ is an $\X$-mutation pair and $\A$ is extension closed,
then the quotient category $\A/\X$ is an $(n+2)$-angulated category.
\end{corollary}

\proof Since $(\A,\A)$ is an $\X$-mutation pair, we have that $\X$ is strongly functorially finite $\A$.
Since $\A$ is extension closed, by \cite[Proposition 2.35]{HLN}, $\A$ is an $n$-exangulated category.
By Lemma \ref{b1} and Lemma \ref{b2}, we obtain $\F:=\F^{\X}=\F_{\X}$. By Theorem \ref{main2}, we have that
$(\A,\F,\s_{\F})$ is a Frobenius $n$-exangulated category whose projective-injective objects are precisely $\X$.
By Theorem \ref{main1}, we get that $\A/\X$ is an $(n+2)$-angulated category.  \qed

\medskip
From the proof of Corollary \ref{cor1}, we have the following.

\begin{example}\label{ex1}
Let $\C$ be an $(n+2)$-angulated category and $\X$ a subcategory of $\C$.
If $(\C,\C)$ is an $\X$-mutation pair, then $(\C,\F^{\X}=\F_{\X},\s_{\F^{\X}=~\F_{\X}})$ is a Frobenius $n$-exangulated category whose projective-injective objects are precisely $\X$.
\end{example}
\smallskip

Now we give a concrete example to explain our second main result.

\begin{example}
This example comes from \cite{L}.
Let $\T=D^b(kQ)/\tau^{-1}[1]$ be the cluster category of type $A_3$,
where $Q$ is the quiver $1\xrightarrow{~\alpha~}2\xrightarrow{~\beta~}3$,
$D^b(kQ)$ is the bounded derived category of finite generated modules over $kQ$, $\tau$ is the Auslander--Reiten translation and $[1]$ is the shift functor
of $D^b(kQ)$. Then $\T$ is a triangulated category. Its shift functor is also denoted by $[1]$.

We describe the
Auslander--Reiten quiver of $\T$ in the following:
$$
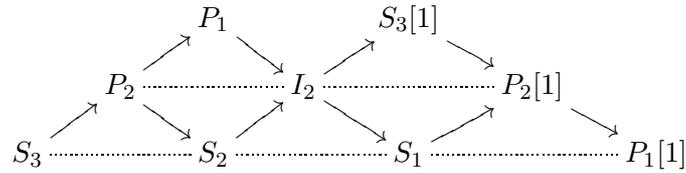
\xymatrix@C=0.6cm@R0.3cm{
&&P_1\ar[dr]
&&S_3[1]\ar[dr]
&&\\
&P_2 \ar@{.}[rr] \ar[dr] \ar[ur]
&&I_2 \ar@{.}[rr] \ar[dr] \ar[ur]
&&P_2[1]\ar[dr]\\
S_3\ar[ur]\ar@{.}[rr]&&S_2\ar[ur]\ar@{.}[rr]
&&S_1\ar[ur]\ar@{.}[rr]&&P_1[1]
}
$$
It is straightforward to verify that $\C:=\add(S_3\oplus P_1\oplus S_1)$ is a $2$-cluster tilting subcategory of $\T$. Moreover, $\C[2]=\C$.  By \cite[Theorem 1]{GKO}, we know that $(\C,[2])$ is a $4$-angulated category.
Let $\X=\add(S_3\oplus S_1)$.
Then the $4$-angle
$$P_1\xrightarrow{~~}S_1\xrightarrow{~~}S_3\xrightarrow{~~}P_1\xrightarrow{~~}P_1[2]$$
shows that $(\C,\C)$ is an $\X$-mutation pair.
By Example \ref{ex1}, we obtain that $(\C,\F,\s_{\F})$ is a Frobenius $n$-exangulated category whose projective-injective objects are precisely $\X$, where $\F:=\F^{\X}=\F_{\X}$.
\end{example}

\section*{Acknowledgements}
This work was carried out when the second author is a postdoctoral fellow at Universit\'{e} de
Sherbrooke. The second author thanks Professor Shiping Liu for his helpful discussions and warm
hospitality. He also wants to thank his colleagues at Universit\'{e} de Sherbrooke for their
help.

\textbf{Yu Liu}\\
School of Mathematics, Southwest Jiaotong University, 610031, Chengdu,
Sichuan, P. R. China\\
E-mail: \textsf{liuyu86@swjtu.edu.cn}\\[0.3cm]
\textbf{Panyue Zhou}\\
College of Mathematics, Hunan Institute of Science and Technology, 414006, Yueyang, Hunan, P. R. China.\\
E-mail: \textsf{panyuezhou@163.com}

\end{document}